\documentclass[11pt]{amsart}  
\usepackage[left=1.3in,right=1.3in,bottom=1.3in]{geometry}  
\usepackage[T1]{fontenc}
\usepackage{amssymb,amsmath,amscd}
\usepackage{amsthm}
\usepackage[parfill]{parskip}
\usepackage{graphicx} 
\usepackage{mathrsfs}
\usepackage{accents}
\usepackage{dsfont}
\usepackage{tikz}
\usepackage{tikz-cd}
\usepackage{mdframed}
\usepackage{caption}
\usepackage{subcaption}

\theoremstyle{definition}
\newtheorem{theorem}{Theorem}
\newtheorem{lemma}[theorem]{Lemma}
\newtheorem{definition}[theorem]{Definition}
\newtheorem{corollary}[theorem]{Corollary}

\newtheorem{example}[theorem]{Example}

\newcommand{\R}{\mathbb{R}}

\newcommand{\bdd}{\partial}

\newcommand{\iso}{\cong}
\newcommand{\id}{\textrm{id}}

\newcommand{\im}{\textrm{im}\,}

\newcommand{\calb}{\mathcal{B}}

\newcommand{\calf}{\mathcal{F}}
\newcommand{\calg}{\mathcal{G}}
\newcommand{\calh}{\mathcal{H}}

\newcommand{\cals}{\mathcal{S}}

\newcommand{\sfe}{\mathsf{E}}
\newcommand{\sff}{\mathsf{F}}
\newcommand{\sfg}{\mathsf{G}}
\newcommand{\sfh}{\mathsf{H}}

\usepackage[style=numeric,
            giveninits=true,     
            uniquename=init      
           ]{biblatex}
\newtheorem*{theorem2}{Theorem}

\title{Unified Origami Kinematics via Cosheaf Homology}
\author{Zoe Cooperband \and Robert Ghrist}

\begin{document}

\begin{abstract}
We establish a novel local-global framework for analyzing rigid origami mechanics through cosheaf homology, proving the equivalence of truss and hinge constraint systems via an induced linear isomorphism. This approach applies to origami surfaces of various topologies, including sheets, spheres, and tori. By leveraging connecting homomorphisms from homological algebra, we link angular and spatial velocities in a novel way. Unlike traditional methods that simplify complex closed-chain systems to re-constrained tree topologies, our homological techniques enable simultaneous analysis of the entire system. This unified framework opens new avenues for homological algorithms and optimization strategies in robotic origami and beyond.
\end{abstract}

\maketitle

\section{Framing}
\label{sec:intro}

The theory of kinematic chains in robotics is classical and well-developed~\cite{lynch2017modern}. Linear and tree-like mechanisms are topologically {\em contactable}, meaning simple methods such as recursive base-to-end equations can be used in analysis~\cite{muller2020n}. In contrast, linkages and closed chain systems have body and joint connections that loop back upon themselves, requiring more advanced analysis methods~\cite{muller2018kinematic}. 

The passage from 1-d locally chain-like mechanisms to locally 2-d surface-like origami configurations adds numerous degrees of complexity~\cite{lynch2017modern, tachi2010geometric}. The impetus for resolving such comes from key applications in contemporary robotics, architecture, meta-materials, and more \cite{fonsecaOverviewMechanicalDescription2022}.

This paper concerns the algebraic properties of \textit{rigid origami surfaces}. We consider an origami surface to consist of rigid faces (or panels) connected by edges (or folds, hinges, joints) in a surface configuration. Boundary components -- cuts or holes -- may be present, but such surfaces will be required to be locally-planar and globally orientable. Corners or vertices where folds intersect can be modeled as spherical linkages, imposing closed loop constraints.

It is common knowledge in the mathematical origami community that there are several first order linear models constituting the kinematic \textit{Pfaffian constraints} of a rigid origami system. Recall that if $q$ is a collection of state variables, a Pfaffian constraint is encoded by a constraint matrix ${\bf A}(q)$ specifying the equation
\begin{equation}
    {\bf A}(q)\dot{q} = 0 .
\end{equation}
Such a linear constraint on velocities is also called a \textit{holonomic} constraint in the state variables, indicating the time-evolution of the system can follow from numerically integrating these state velocities $\dot{q}$ \cite{zhu2022review}. These Pfaffian constraints indicate an infinitesimal folding path in $q$ to first order. This does not capture full folding behavior; for instance, the flat state of an origami sheet can be highly degenerate. In this paper we will not consider higher-order analysis, such as the zero-area criterion for second-order foldability \cite{demaineZeroAreaReciprocalDiagram2016} to detect such pathologies.

There are several approaches to modeling rigid origami kinematics, each assigning velocity variables to different cells. In the \textbf{spatial body model}, each panel is assigned three angular and three linear velocity DOF, these restricted by five degrees of constraints at hinges. In the \textbf{hinge model}, the angular velocities of the hinges $\dot{\theta}$ are free variables, constrained around each vertex. In a third model, the \textbf{truss model}, bars are added within faces so that the linear velocities of vertices suffice to describe the full origami motion: see Figure~\ref{fig:overview}.

\begin{figure}
    \centering
    \includegraphics{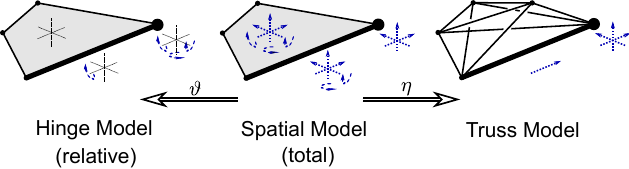}
    \caption{Kinematic data assignments of different origami kinematic models are pictured over a sample face, edge, and vertex. The hinge and spatial models have ``complementary'' cellular geometric data assignments in the sense that one has data where the other lacks it.}
    \label{fig:overview}
\end{figure}

All three of these kinematic origami models seem clearly interconnected and their solutions transferable by-hook-or-by-crook. A translation between spatial and hinge velocity coordinates is already a common procedure in the robotics literature~\cite{jain2010robot, jain2011graph}. 

\subsection{Brief Statement of Results}
\label{sec:inbrief}

The principal result of this paper is a formal homological framework encompassing all three models, with explicit and natural algebraic equivalences between them. These equivalences hold not merely in the linear-chain or tree settings, but in the full closed-chain and surface systems, without recourse to cutting and unfolding. For brevity, we state simplified versions of the results using the notation to be developed in Sections~\ref{sec:cosheaves}-\ref{sec:relations}.

\begin{theorem2}[\ref{thm:holes}]
    Fix $(X, p)$ an arbitrary oriented origami surface (possibly with boundary). The isomorphism $\vartheta^+|_{\ker\iota_*}$ sends solutions of the hinge model satisfying $6\dim (H_1 X)$ additional constraints (encoded by $\iota_*$), to solutions of the spatial model up to global translation and rotation. 
\end{theorem2}

Formally, the map $\vartheta$ is a \textit{connecting homomorphism} and $\vartheta^+$ is its \textit{Moore-Penrose pseudo-inverse}. This map $\vartheta^+$ takes systems of infinitesimal hinge motions and converts these to minimal rigid body motion of the origami panels in the least-squares sense. The vector space $H_1 X$ has a basis of homology loops of $X$ - loops over the surface that cannot be contracted to a point (such as a loop drawn around a puncture). The map $\iota_*$ in Theorem~\ref{thm:holes} measures the spatial alignment error around each loop, and $\ker\iota_*$ is the space of ``globally consistent solutions'' to the hinge model. We demonstrate the abstract result on serial chains in Section~\ref{sec:classical}.

The connection to the truss kinematic model is completed by:

\begin{theorem2}[\ref{thm:eta_hom}]
    Suppose $(X, p)$ is any oriented origami surface. The isomorphism $\eta$ maps solutions of the spatial model to solutions of the truss model.
\end{theorem2}

The map $\eta$ above is defined in Lemma~\ref{lem:eta_def} and translates motion of a rigid panel to the motion of vertices along the panel boundary. Theorem~\ref{thm:eta_hom} proves that both variables and constraints are translated faithfully. The domains and codomains of these maps $\vartheta$ and $\eta$ are sketched in Figure~\ref{fig:overview}. The above two theorems combine to yield:

\begin{corollary}\label{thm:full}
    Suppose $(X, p)$ is an origami surface. Then $\eta \vartheta^+|_{\ker \iota_*}$ is a linear isomorphism between solutions of the hinge model and the truss model up to global translation and rotation.
\end{corollary}

This result is significant. It is often assumed that the hinge and truss models provide equivalent Pfaffian constraints for rigid origami. To the authors' knowledge, Corollary~\ref{thm:full} is the first explicit proof of this result (though distilled from the first author's  thesis~\cite{cooperbandCellularCosheavesGraphic2024}). More important than the existence of a formal proof is its nature: by setting up the proper structures, the equivalence between models is a simple matter of reading off a natural isomorphism that arises not from careful tweaking but from universal constructs in homological algebra.

Homological methods in applied mathematics are a topic of recent interest, though with historical antecedents in the 20th century, including abstract formulations of Kirchhoff's laws ~\cite{davies1981kirchhoff,smale1972mathematical}, fixed point theorems in Economics \cite{urai2010fixed}, index theory in dynamical systems \cite{milnor1963morse}, and spline geometry \cite{billera1988homology}. 
The most well-known contemporary uses of homological methods focus on persistent homology in topological data analysis, with applications to genetics \cite{nicolau2011topology, chan2013topology}, neuroscience \cite{sizemore2019importance, yoon2024tracking}, material science~\cite{kramar2013persistence, onodera2019understanding}, sensor networks~\cite{de2007coverage}, path planning~\cite{bhattacharya2015persistent}, and more. The specific methods of this paper (cellular cosheaves and their dual sheaves) are more recent still, applied to network coding \cite{ghrist2011network}, distributed optimization \cite{hansen2019distributed}, consensus \cite{hansen2020laplacians, RiessLatticeTarski}, social information dynamics \cite{hansen2021opinion, riess2022diffusion}, and neural networks \cite{hansen2020sheaf, bodnar2022neural}.

The homological techniques that fuel this paper are not well-known in all communities. To provide the reader with a quick summary without interrupting the flow of the paper, the basics of what is needed is collected in the Appendix (expressed in the language of l inear algebra as opposed to the more general module-theoretic approach of standard texts). The reader familiar with cellular sheaves or cosheaves can skip this and proceed with the following recap of the classical methods and notation for simple origami structures. 

\subsection{Serial Chain Methods}
\label{sec:classical}
Theorem~\ref{thm:holes} is abstract; in simple settings this describes a well known procedure. To demonstrate, we detail the Theorem over an origami in serial chain configuration, pictured in Figure~\ref{fig:serial}.

\begin{figure}
    \centering
    \includegraphics[scale = 0.8]{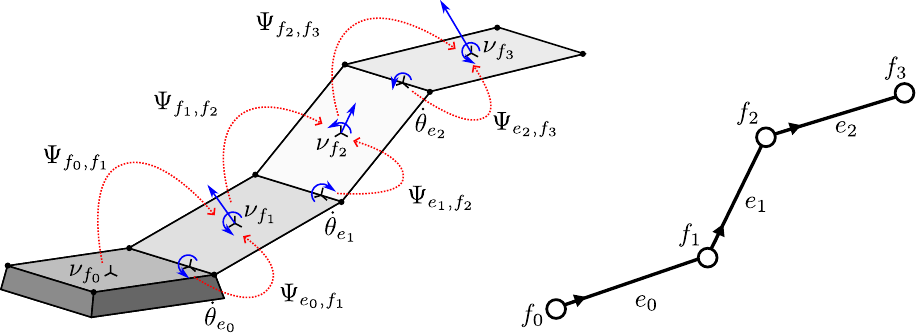}
    \caption{An origami serial chain. Spatial velocities, angular velocities, and rigid body operators are pictured (left) . The dual graph (right) is directed along the recursive equation~\eqref{eq:recurrence}. This graph is acyclic, indicating a open-chain system.}
    \label{fig:serial}
\end{figure}

\begin{definition}[Origami]
\label{def:origami}
    A {\it rigid origami surface} consists of panels/faces $F = \{f\}$ joined at their boundaries by straight creases/ hinges/edges $E = \{e\}$ and vertex joints $V = \{v\}$. These abstract cells join to form a cellular surface $X$ with a transitive incidence relation $v \lhd e$ and $e\lhd f$ between adjacent/incident cells. The surface $X$ is {\it realized} by a map $p: X \to \R^3$ that assigns each cell a {\it centroid} coordinate in $\R^3$. The pair $(X,p)$ is a (non-degenerate) {\it origami surface}\footnote{ Our definition of an origami surface is general and allows for systems not typically considered as origami. For instance, face panels $\{f\}$ need not be flat; if ``wavy'' faces are undesired this additional constraint can be added afterward. Moreover, this Definition~\ref{def:origami} does not specify that the origami surface originate from a flat state or that interior angles add to $2\pi$. Self-intersecting panels are allowed.
} if $p$ satisfies
    \begin{equation}\label{eq:nondegen}
        \text{dim span} \{p_v : v \lhd c\} = \dim c + 1,
    \end{equation}
meaning coordinates of each $i$ cell span an $(i+1)$-dimensional subspace of $\R^3$ (edges and faces have non-zero length and area).
\end{definition}

To explicitly localize state variables, let $\sff$ denote a reference frame based at a point in Euclidean space $p_\sff\in \R^3$ fixed to a rigid body $f$. A {\it spatial velocity} $\nu_\sff$ in the frame $\sff$ on the body $f$ is a $6$-dimensional real vector $[\omega,\beta]_\sff^\top$ where $\omega$ is the angular velocity and $\beta$ is the linear velocity. Spatial velocity vectors span the six-dimensional lie algebra $se(3)$ of the special Euclidean group $SE(3)$. When $f$ has only the one frame $\sff$ of note, we write $\nu_f$ for $\nu_\sff$. Although we will not use this representation, it is common to represent a spatial velocity in {\it homogeneous coordinates} where $\nu$ is represented by the $4\times 4$ matrix
\begin{equation}
\label{eq:homogeneous}
    \begin{bmatrix}
        \omega_\times & \beta \\
        {\bf 0} & 1
    \end{bmatrix}.
\end{equation}
Here $\omega_\times$ refers to the cross product operator of $\omega$, the $3\times 3$ anti-symmetric matrix
\begin{equation}
\label{eq:skew_symm}
    \begin{bmatrix}
        \omega_1 \\
        \omega_2 \\
        \omega_3
    \end{bmatrix}_\times
    =
    \begin{bmatrix}
        0 & -\omega_3 & \omega_2 \\
        \omega_3 & 0 & -\omega_1 \\
        -\omega_2 & \omega_1 & 0
    \end{bmatrix}
\end{equation}
where $\omega_i$, $i=1,2,3$ are the coordinates of the angular velocity. Here $\omega_\times a$ is the cross product $\omega\times a$ for $a\in \R^3$.

There is a simple relationship between velocities measured at different points across a rigid body. We define {\it rigid body operators} $\Psi$ as $6\times 6$ matrices acting on spatial velocities
\begin{small}\begin{equation}
    \label{eq:rigid_body_op_kinematic}
    \Psi_{\sff,\sfg}
    \begin{bmatrix}
        \omega \\
        \beta
    \end{bmatrix}_\sff
    =
    \begin{bmatrix}
        {\bf R}_{\sff,\sfg} & {\bf 0} \\
        {\bf 0} & {\bf R}_{\sff,\sfg}
    \end{bmatrix}
    \begin{bmatrix}
        {\bf I} & {\bf 0} \\
        (p_\sff - p_\sfg)_\times & {\bf I}
    \end{bmatrix}
    \begin{bmatrix}
        \omega \\
        \beta
    \end{bmatrix}_\sff
    =
    \begin{bmatrix}
        {\bf R}_{\sff,\sfg} \omega \\
        {\bf R}_{\sff,\sfg} (\omega \times (p_\sfg - p_\sff) + \beta)
    \end{bmatrix}_\sfg,
\end{equation}\end{small}
transferring a velocity at $\sff$ to the velocity at $\sfg$. Here ${\bf R}_{\sff,\sfg}$ is a rotation matrix between the two frames, which we assume to be the identity matrix in this paper to simplify infinitesimal analysis. It quickly follows that rigid body operators have the algebraic composition property $\Psi_{\sfg,\sfh}\Psi_{\sff,\sfg} = \Psi_{\sff,\sfh}$ and inverse property $(\Psi_{\sff,\sfg})^{-1} = \Psi_{\sfg,\sff}$.

A serial chain of folding panels is a chain of rigid bodies $\{f_0, \dots, f_n\}$, linked by single DOF joints $\{e_1, \dots, e_n\}$ modeling a robotic arm. Each body $f_i$ and hinge $e_i$ is assigned a frame $\sff_{f_i}$ and $\sfe_{e_i}$, the latter located along the hinge line of action. The base of the arm $f_0$ is fixed in space while each hinge attains an instantaneous angular velocity $\dot{\theta}_{e_i} = \dot{\theta}_i$. Knowing these hinge angular velocities, the body velocities $\nu_{f_i} = \nu_i$ can be found by the well known recursive equation
\begin{equation}
\label{eq:recurrence}
    \nu_{i+1} = \Psi_{f_i, f_{i+1}} \nu_i + \Psi_{e_i, f_{i+1}} \iota_{e_i}\dot{\theta}_i
\end{equation}
with initial conditions $\nu_0 = 0$~\cite{jain2010robot, rodriguez1992spatial}. In \eqref{eq:recurrence} the frames at $f_i,f_{i+1},e_i$ are assumed to be attached to the cells. The maps $\Psi_{f_i, f_{i+1}}$ and $\Psi_{e_i, f_{i+1}}$ above are spatial operators that transfer a spatial velocity at one frame to the spatial velocity at another. The map $\iota_{e_i}$ is known as the {\it hinge operator} and embeds angular velocity values $\R$ into the subspace parallel to the hinge axis. For instance, if the hinge axis aligns along the first $x$ axis, then $\iota_{e_i}$ is the $6\times 1$ matrix $[1,0,0,0,0,0]^\top$.

Let $\iota = {\bf Diag}(\iota_{e_i})$ denote the size $6n \times n$ block diagonal {\it global hinge matrix} with matrices $\iota_{e_i}$ along the diagonal. Furthermore let
\begin{equation}
\label{eq:rigid_matrix}
\Psi = 
    \begin{bmatrix}
        \Psi_{e_1,f_1} & {\bf 0} & {\bf 0} & \hdots & {\bf 0} \\
        \Psi_{e_1,f_2} & \Psi_{e_2,f_2} & {\bf 0} & \hdots & {\bf 0} \\
        \Psi_{e_1,f_3} & \Psi_{e_2,f_3} & \Psi_{e_3,f_3} & \hdots & {\bf 0} \\
        \vdots & \vdots & \vdots & & \vdots \\
        \Psi_{e_1,f_n} & \Psi_{e_2,f_n} & \Psi_{e_3,f_n} & \hdots & \Psi_{e_n,f_n}
    \end{bmatrix}
\end{equation}
be the {\it global rigid body operator} and let ${\bf D} = \bf{I} \Psi \iota$ where ${\bf I}$ is an identity matrix; the operator~\ref{eq:rigid_matrix} adds all appropriately transformed spatial hinge velocities from the base of the serial chain to the tip. The combination of all recursive equations~\eqref{eq:recurrence} for all body indices $i$ combine to form the matrix equation $\nu = {\bf D} \dot{\theta}$ where the global vectors $\nu$ and $\dot{\theta}$ are comprised of stacked local variables $\nu_{f_i}$ and $\dot{\theta}_{e_i}$ (excluding the term $\nu_{f_0} = 0$).

We observe that equation~\eqref{eq:recurrence} can be rearranged into the form
\begin{equation}
\label{eq:connecting}
    \iota_{e_i}\dot{\theta}_{e_i} = \Psi_{f_{i+1},e_i} \nu_{f_{i+1}} - \Psi_{f_i, e_i} \nu_{f_i}.
\end{equation}
Encoding equation~\eqref{eq:connecting} at all hinges, the global matrix $\Psi$ has inverse
\begin{equation}
\label{eq:rigid_inv}
   \Psi^{-1} =
    \begin{bmatrix}
        \Psi_{f_1,e_1} & {\bf 0} & {\bf 0} & \hdots & {\bf 0} \\
        -\Psi_{f_1,e_2} & \Psi_{f_2,e_2} & {\bf 0} & \hdots & {\bf 0} \\
        0 & -\Psi_{f_2,e_3} & \Psi_{f_3,e_3} & \hdots & {\bf 0} \\
        \vdots & \vdots & \vdots & & \vdots \\
        {\bf 0} & {\bf 0} & {\bf 0} & \hdots & \Psi_{v_n,e_n}
    \end{bmatrix}
\end{equation}
and we may rearrange the governing matrix equation to $\iota \dot{\theta} = \Psi^{-1} {\bf I} \nu$.

The operator ${\bf D}$ is injective but not invertable. Its {\it Moore-Penrose pseudoinverse} ${\bf D}^+ = \iota^\top \Psi^{-1} {\bf I}$ is the closest generalization of the inverse. Here ${\bf D}^+$ is surjective and whenever $\nu$ satisfies the equation $\nu = {\bf D} \dot{\theta}$ it follows that $\dot{\theta} = {\bf D}^+ \nu$.

Here are the critical observations. The matrix of rigid body operators~\eqref{eq:rigid_inv} is a {\it boundary operator}, where the boundary edges of each face $e_i,e_{i+1} \lhd f_i$ are reflected in each column of $\Psi^{-1}$. The full matrix ${\bf D}^+$ is a {\it connecting homomorphism}, a composition of operations taking the preimage by a surjective map ${\bf I}$, then applying the boundary operator $\Psi^{-1}$ and finally taking the preimage of an injective map $\iota$. The kinematic solutions $(\nu, \dot{\theta})$ satisfying $\nu = {\bf D} \dot{\theta}$ must be homology classes $\nu\in H_2 \cals$ and $\dot{\theta}\in H_1\calh$ with connecting homomorphism ${\bf D}^+: H_2 \cals \to H_1 \calh$ (all to be defined later). This formalism allows different algebraic models to be integrated with each other as well as with the underlying spatial topology.

While in this serial chain example the rigid body operator matrix~\eqref{eq:rigid_matrix} and its inverse~\eqref{eq:rigid_inv} are well known~\cite{jain2010robot}, dividends are evident when moving to more complex topologies. The connecting map (here ${\bf D}^+$) is easily computed for any rigid body system, then with pseudoinverse (${\bf D}^{++}$) dictating the transition from angular hinge variables to spatial $\R^6$ variables.

\section{Rigid Origami Models}
\label{sec:classical_models}

In this section we derive the constraints encoding the different origami models. These serve as a starting points to the cosheaf representations of the underlying algebraic structures.

\subsection{Hinge Model}
\label{sec:hinge_model}

The first central method in kinematic analysis is the hinge model, encoding the {\it loop closure constraint} developed by Belcastro and Hull \cite{belcastroModellingFoldingPaper2002} and Tachi \cite{tachiSimulationRigidOrigami2009} among other authors.

Traversing along a loop around a vertex of an origami surface, each crease constitutes a change in local coordinates between adjacent faces. Composing these transformations, the frame is conserved upon return to the start of the loop. From edge to edge, the frame is transformed by a $3\times3$ orthogonal rotation matrix ${\bf R}_i(\theta_i)$ about the axis parallel to the edge $e_i$ by angle $\theta_i$, these rotation matrices being elements of the Lie group $SO(3)$. In a loop around a vertex $v$ with $m$ incident edges, these rotation matrices compose to the identity transformation:
\begin{equation}\label{eq:rot_id}
    {\bf R}(\theta_0, \dots, \theta_{m-1})
    = {\bf R}_{m-1}(\theta_{m-1}) \dots {\bf R}_1(\theta_1){\bf R}_0(\theta_0)
    = {\bf I}
\end{equation}

Taking the derivative of equation~\eqref{eq:rot_id} with respect to time, one finds
\begin{equation}\label{eq:rot_deriv}
    \frac{d{\bf R}}{dt}
    = \sum_{i=0}^{m-1} \frac{\partial {\bf R}}{\partial \theta_i} \frac{\partial \theta_i}{\partial t}
    = {\bf 0}
\end{equation}
where the terms $\partial \theta_i / \partial t = \dot{\theta_i}$ are the instantaneous angular velocities of the hinges $e_i$. Because ${\bf R} = {\bf I}$ each matrix $\partial {\bf R} / \partial \theta_i$ is skew symmetric of form~\eqref{eq:skew_symm}, these being elements of the Lie algebra $so(3)$~\cite{tachiSimulationRigidOrigami2009}. 

The matrix equation~\eqref{eq:rot_deriv} condenses to three independent equations of off-diagonal terms, each fixing one directional rotation component to be zero. These constraints must be satisfied at every interior vertex of the origami, leading to the size $3|V_\text{in}| \times |E_\text{in}|$ Jacobian matrix $\bdd_\calh$ of the constraints when $X$ where $V_{\text in}$ and $E_{\text in}$ are the sets of internal vertices and edges of the surface\footnote{Internal edges are incident to two faces and internal vertices do not lie on the boundary of $X$.}. This matrix constitutes the hinge model \cite{tachiSimulationRigidOrigami2009}.

If the sheet is punctured and non-simply connected, each puncture induces 6 constraint equations per hole \cite{tachi2010geometric}. These not only include the three orientation-preserving constraints in equation~\eqref{eq:rot_deriv}, but three translational constraints as well. We prove more generally that 6 constraints are required around each generating homology loop $H_1 X$.

\subsection{Spatial Model}
\label{sec:spatial_model}
The constraints on spatial velocities in origami have already been touched upon in Section~\ref{sec:classical}. Suppose that $e\lhd f,g$ are two faces sharing a common hinge $e$. Fixing two reference frames $\sff$ and $\sfg$ at $f$ and $g$, there is a Euclidean transformation ${\bf Q}(q)\in SE(3)$ between the two, dependent on state variables $q$. It is not hard to see that the family of transformation ${\bf Q}(q)$ encodes a one parameter submanifold of $SE(3)$ isomorphic to the circle group $S^1$.

The previous equation~\eqref{eq:connecting} embodies this Pfaffian constraint. For an edge with endpoints $u,v \lhd e$ let $l_e$ denote the unit normalization of the vector $\pm (p_u - p_v)$ (This convenient notation will only be used when the sign does not matter). If $\nu = [\omega,\beta]^\top$ is a spatial velocity then let $\pi_e$ be the size $5\times 6$ matrix projecting on to the subspace $(\text{span}\{[l_e,0]^\top\})^\perp$. Here $\omega$ is mapped to its projection $[\omega]$ on the subspace $l_e^\perp$ orthogonal to the hinge axis $l_e$, while $\beta$ is left unchanged. If $\nu_f$ and $\nu_g$ are face spatial velocities, they must satisfy the equation
\begin{equation}
\label{eq:spatial_constraint}
    \pi_e(\Psi_{f,e} \nu_f - \Psi_{g,e} \nu_g) = {\bf 0}.
\end{equation}
Requiring equation~\eqref{eq:spatial_constraint} to be satisfied everywhere leads to a global $5|E_\text{in}| \times 6|F|$ constraint matrix $\bdd_\cals$ on spatial velocities. In the serial chain example of Section~\ref{sec:classical}, it is evident that spatial velocities satisfying equation~\eqref{eq:connecting} must satisfy equation~\eqref{eq:spatial_constraint}.

\subsection{Truss Model}
\label{sec:truss_model}

The truss model is a familiar approach to deriving kinematics, where vertices are free instead of hinges. The length of an edge with endpoints $u,v\lhd e$ is $\|p_u - p_v\|$ which we assume remains constant. Taking the derivative of this constraint equation at $e$ we find the Pfaffian constraint
\begin{equation}\label{eq:edge_constraint}
    \frac{d}{dt} \| p_u - p_v \| = \frac{d}{dt} \sqrt{\langle p_u - p_v, p_u - p_v \rangle} = \left\langle \frac{p_u - p_v}{\|p_u - p_v\|}, \frac{d p_u}{dt} - \frac{d p_v}{dt} \right\rangle = 0.
\end{equation}
in terms of six variables: the three separate directional components of $d p_u/ dt$ and $d p_v / dt$.

\begin{figure}
    \centering
    \includegraphics[scale = 1.2]{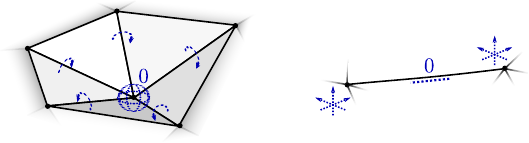}
    \caption{The degrees of freedom and constraints of the hinge model (left) and truss model (right) are displayed. The rotational velocities sum to zero around each vertex (left) and the length of the edge changes by zero (right). The data assignments of Figure~\ref{fig:overview} are divided into degrees of freedom and constraints.}
    \label{fig:constraints}
\end{figure}

It is not sufficient to require the constraint~\eqref{eq:edge_constraint} be satisfied over every edge; there are often unwanted degrees of freedom warping faces in and out of plane. To ensure faces remain rigid we embellish the surface $X$ with ``enough'' edges and vertices. For each face $f$ we add additional vertices $f'$ such that $f'$ is realized out of the plane of the other points $\{p_v: v\lhd f\}$. We further add edges to $X$ ensuring that $\{v: v\lhd f\} \cup \{f'\}$ is a fully connected sub-graph\footnote{
This is sufficient to ensure that each sub-linkage associated to a face $f$ is rigid. Many of these edges may be redundant.
}. The resulting cell complex is a {\it stiffened linkage} $X'$, sketched in Figure~\ref{fig:overview} (right).

Solutions of the truss model must satisfy equation~\eqref{eq:edge_constraint} at every edge of the origami in $X'$, leading to a size $|E'| \times 3|V'|$ matrix of total constraints (in the number of edges and vertices of $X'$). A vector in the null-space of this matrix leaves each sub-linkage (associated to a face) rigid.
\section{Cosheaf Origami}
\label{sec:cosheaves}

The origami constraints developed in Section~\ref{sec:classical_models} can be written in a more abstract and concise topological language using {\em cellular cosheaves}: see the Appendix for an introduction to the theory. 

\subsection{The Hinge Cosheaf}
\label{sec:hinge_cosheaf}

Here we show that the constraints in the rotating hinge model of Section~\ref{sec:hinge_model} can be implemented by a cosheaf.

\begin{definition}[Hinge cosheaf]
    Let $(X,p)$ be an origami surface. The {\it hinge cosheaf} $\calh$ has vertex stalks $\calh_v = so(3)$ the full space of $3\times 3$ skew symmetric matrices and edge stalks $\calh_e = \R = \text{span}\{\dot{\theta}_e\} $ isomorphic to axis aligned spatial angular velocities $\text{span}\{l_e\} \subset so(3)$. The cosheaf $\calh$ has trivial stalks over faces. Between edges and vertices the cosheaf extension map $\calh_{e\rhd v}$ sends the unit hinge velocity $1\in \calh_e$ to $l_e$ in the larger space\footnote{
    To an edge $u,v\lhd e$, both extension maps $\calh_{e\rhd v}$ and $\calh_{e\rhd u}$ are the ``same,'' sending $1$ to the same multiple $l_e$.
    }.
\end{definition}
The hinge cosheaf characterizes the hinge-type constraints of equation~\eqref{eq:rot_deriv}. Here $1$-chains $C_1 \calh$ entail the degrees of freedom of hinges about an axis of rotation while $0$-chains $C_0 \calh$ track angular constraints at vertices.

\begin{lemma}
    Elements of $H_1 \calh$ encode solutions to the rotating hinge model discussed in Section~\ref{sec:hinge_model}.
\end{lemma}
\begin{proof}
    The homology space $H_1 \calh = \ker \bdd_\calh$ of the hinge cosheaf $\calh$ consists of the chains with zero boundary in $C_0 \calh$, meaning the sum of the rotations about every vertex $v$ is zero. The constraint that a homology class $\dot{\theta}$ satisfies $\bdd_\calh (\dot{\theta})_v = 0$ at every (interior) vertex is the same as equation~\eqref{eq:rot_deriv} where the skew symmetric matrix $\partial{\bf R} / \partial\dot{\theta_e}$ is the cosheaf map $\calh_{e\rhd v}$.
\end{proof}

\subsection{The Spatial Cosheaf}
\label{sec:kinematic_cosheaf}
We develop a second, more complex model for origami kinematics where cells are modeled as rigid bodies in $\R^3$. The {\it spatial origami cosheaf} over an origami surface $(X,p)$ encodes the degrees of freedom of the rigid body motion of faces.

\begin{definition}[Spatial cosheaf]
\label{def:kinematic}
    Let $(X, p)$ be an origami surface. The {\it spatial cosheaf} $\cals$ has of stalks of spatial vectors $\cals_f = \R^6\iso se(3)$, $\cals_e \iso \R^5 \iso se(3) / \text{span}\{(l_e)_\times\}$, and $\cals_v \iso \R^3 \iso se(3) / so(3)$ over faces, edges, and vertices (as quotient vector spaces). The extension map from a face $f$ to edge $e$ takes value
    \begin{small}\begin{equation*}
        \cals_{f\rhd e}
        \begin{bmatrix}
            \omega \\
            \beta
        \end{bmatrix}_f := 
        \pi_e \Psi_{f,e}
        \begin{bmatrix}
            \omega \\
            \beta
        \end{bmatrix}_f
        = 
        \begin{bmatrix}
            [\omega] & \\
            \beta + \omega_\times(p_e - p_f)
        \end{bmatrix}_e
    \end{equation*}\end{small}
where $\pi_e$ sends $\omega$ to the quotient vector $[\omega] = \omega + \text{span}\{l_e\}$, in matrix form considered as an element of the quotient space $so(3)/\text{span}\{(l_e)_\times\}$; in matrix form this is the same map as in line~\eqref{eq:spatial_constraint}. To vertices we have
    \begin{small}\begin{equation*}
        \cals_{e\rhd v} 
         \begin{bmatrix}
            [\omega] \\
            \beta
        \end{bmatrix}_e :=
        \pi_v \Psi_{e,v}
        \begin{bmatrix}
            [\omega] \\
            \beta
        \end{bmatrix}_e
        = 
        \begin{bmatrix}
            [0] \\
            \beta + \omega_\times(p_v - p_e)
        \end{bmatrix}_v
    \end{equation*}\end{small}
with $[0]$ the trivial zero quotient class. The extension map $\cals_{e\rhd v}$ is well defined because $l_e$ is parallel with the vector $(p_v - p_e)$ causing
    \begin{equation}\label{eq:quotient_reduces}
        [\omega]_\times(p_v - p_e) = (\omega_\times + \text{span}\{(l_e)_\times\})(p_v - p_e) = \omega_\times(p_v - p_e).
    \end{equation}
\end{definition}
It is easy to see that $\dim \cals_f=6$, $\dim \cals_e = 5$, and $\dim \cals_v = 3$ for each face, edge, and vertex. From a choice of spatial velocity vector $\nu_f = [\omega, \beta]^\top_f$ at the face $f$ frame we can recover the induced spatial velocity of any cell along the boundary of $f$. For a vertex $v\lhd e \lhd f$ incident to the face, its change in position is computed directly by the lower $\R^3$ linear velocity component of the spatial vector

\begin{small}\begin{equation}
\label{eq:kin_composition}
    (\cals_{e \rhd v} \cals_{f\rhd e}) x_f
    =
    \cals_{f \rhd v}
     \begin{bmatrix}
        \omega \\
        \beta
    \end{bmatrix}_f
    =
    \begin{bmatrix}
        [0] \\
        \beta + \omega_\times(p_v - p_f)
    \end{bmatrix}_v.
\end{equation}\end{small}
For notation purposes, we let $\eta_{f \rhd v}: \R^6 \to \R^3$ denote the projection of the rigid body operator $\Psi_{f,v}$ to the bottom three linear velocity coordinates
\begin{equation}\label{eq:eta}
    \eta_{f \rhd v}(\nu_f) := \beta + \omega_\times(p_v - p_f),
\end{equation}
the lower terms in line~\eqref{eq:kin_composition}.

We form the spatial cosheaf chain complex $C \cals$, where we are mainly interested in its second homology $H_2 \cals$.
\begin{itemize}
    \item The space $C_2 \cals$ consists of free (angular and spatial) velocity assignments to all faces. In robotics these are typically called {\it external velocity coordinates} (in contrast to internal hinge angular-velocity coordinates --- elements $\dot{\theta}\in C_1 \calh$ of the hinge cosheaf).
    \item The space $C_1 \cals$ consists of constraints between rigid panels. Between two neighboring faces $f,g\rhd e$, the value over the edge measures the failure for the kinematics of $f$ and $g$ to align.
\end{itemize}

\begin{lemma}
    Elements of $H_2 \cals$ are first order kinematic degrees of freedom of the origami (of any oriented manifold topology $X$).
\end{lemma}
\begin{proof}
    Suppose that $\nu\in H_2 \cals = \ker \bdd_\cals$. This means that $\cals_{f\rhd e} \nu_f = \cals_{g\rhd e}\nu_g$ at every two panels sharing a common hinge $e\lhd f,g$, meaning equation~\eqref{eq:spatial_constraint} is satisfied. The induced velocities $\Psi_{f,e} \nu_f$ and $\Psi_{g,e} \nu_g$ at the edge frame $e$ must be equal modulo a multiple of $l_e$, an angular rotation along the hinge axis of $e$. This constraint is descriptive, meaning that all kinematic degrees of freedom themselves are elements of $H_2 \cals$.
\end{proof}

From Section~\ref{sec:spatial_model}, $H_2 \cals$ can be associated with the kernel of the size $5|E_{\text in}| \times 6|F|$ matrix $\bdd_\cals$ and is the first order solution space of an {\it spatial model} for origami. The cosheaf $\cals$ simply encodes the kinematic constraints in a compact algebraic form.

\begin{figure}
    \centering\includegraphics[scale = 1.4]{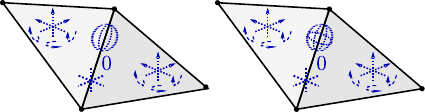}
    \caption{The analogues to Figure~\ref{fig:constraints} for the spatial model $\cals$ (left) and rigid model $\calb$ (right) are pictured. [Left] Faces have six degrees of freedom and are subject to three translational and two rotational constraints over edges where only hinged rotation is permitted. Over the pictured system the homology space $H_2 \cals$ is seven-dimensional, the kernel of the associated size $5 \times 12$ boundary matrix $\bdd_\cals$. Six generators of $H_2 \cals$ characterize global degrees of freedom while the last generator is the hinge action. [Right] As a unified rigid body, edges have six degrees of constraints locking adjacent faces in place relative to each other.}
    \label{fig:euclidean}
\end{figure}

\section{Homology Relations Between Origami Models}
\label{sec:relations}
Now that we have defined three models for origami kinematics (the hinge, truss, and spatial origami models) we now relate the three using homology theory.

\subsection{Relating the Hinge and Spatial Origami Systems}
\label{sec:hinge_kinematic}

Investigating the hinge cosheaf $\calh$ and the spatial cosheaf $\cals$ we see that these are cosheaves of complimentary data. Where $\calh$ describes axis aligned rotation at edges, the cosheaf $\cals$ detects all rigid body motions {\it except} for those axial rotations. Similarly the data of $\calh$ includes pure rotations at vertices while $\cals$ instead describes pure translations at vertices.

To tie $\calh$ and $\cals$ together we define the rigid body cosheaf $\calb$ with additional constraints. All cells are assigned a full six degrees of kinematic data.

\begin{definition}
Let $(X, p)$ be a origami surface. Let the stalks of the {\it rigid body cosheaf} $\calb$ be $\calb_f = \calb_e = \calb_v = \R^6 \iso se(3)$ comprised of spatial velocity vectors. For two cells $c\lhd d$ in $X$ we the extension maps are $\calb_{d,c} = \Psi_{d,c}$.
\end{definition}

We are mainly interested in the second homology $H_2 \calb$, equivalent to the kernel of a size $6|E_{\text in}| \times 6|F|$ matrix $\bdd_\calb$. For $\nu\in H_2 \calb$ to be a homology cycle it is required that $\Psi_{f,e} \nu_f = \Psi_{g,e} \nu_g$ over neighboring faces $e\lhd f,g$. These constraints cascade throughout the surface, meaning $H_2 \calb$ consists of only the six global degrees of translation and rotation. One can picture $H_2 \calb$ as the kinematics of the origami with all of the hinges glued stiff so the structure is reduced to a trivial system.

\begin{lemma}\label{lem:constant_iso}
    There are isomorphisms of cosheaf chain complexes $C \calb \iso C \overline{se(3)}$ inducing isomorphisms on homology.
\end{lemma}
\begin{proof}
    The constant cosheaf $\overline{se(3)}$ (see example~\ref{ex:constant}) can be thought of as the rigid body cosheaf $\calb$ where all cell frames are located at the origin $0$. Rigid body operators $\Psi_{0,c}$ at cells $c$ are isomorphisms between stalks sending
    \begin{equation}
        \Psi_{0,c}\begin{bmatrix}
            \omega \\
            \beta
        \end{bmatrix}_0
        =
        \begin{bmatrix}
            {\bf I} & {\bf 0} \\
            (0 - p_c)_\times & {\bf I}
        \end{bmatrix}
        \begin{bmatrix}
            \omega \\
            \beta
        \end{bmatrix}_0
    =
    \begin{bmatrix}
        \omega \\
        \omega \times p_c + \beta
    \end{bmatrix}_c.
    \end{equation}
    Clearly $\Psi_{0,c}{\bf I} = \calb_{d\rhd c} \Psi_{0,d} = \Psi_{0,c}$ over incident cells $c\lhd d$, where ${\bf I} = \overline{se(3)}_{c,d}$ is the identity transformation. This proves that the family of rigid body operators $\Psi_{0,-}$ form a cosheaf isomorphism between the constant cosheaf $\overline{se(3)}$ and $\calb$.
\end{proof}
Lemma~\ref{lem:constant_iso} indicates that choices for cell centroids $p_c$ amount to a choice of bases across all kinematic assignments. For any oriented surface $X$, the second homology of the space $H_2 X$ is one-dimensional consisting of the orientation class. By example~\eqref{ex:constant} we know $H_2 \calb \iso se(3)$, then Lemma~\ref{lem:constant_iso} confirms $H_2\calb \iso H_2 \overline{se(3)} \iso se(3)$ consists of only the trivial global kinematic motions.

The following Lemma proves how the various cosheaf models constructed so far relate, namely that the cell assignments of the hinge and spatial cosheaves $\calh$ and $\cals$ combine to the trivial cell assignments of $\calb$.

\begin{lemma}\label{lem:defined_exact}
    The following is a short exact sequence of cosheaves
    \begin{equation}\label{eq:origami_ses}
    0 \to \calh \xrightarrow{\iota} \calb \xrightarrow{\pi} \cals \to 0
    \end{equation}
    where $\iota$ embeds stalks in the larger space and $\pi$ is the quotient map of stalks seen in Definition~\ref{def:kinematic}.
\end{lemma}
\begin{proof}
We must prove that for every triple of incident cells $v\lhd e\lhd f$ the following diagram commutes with exact rows.
    \begin{equation}\label{eq:comm_diag}
        \begin{tikzcd}
            & 0 \ar[r] \ar[d] & \calb_f \ar[d] \ar[r, "\id"]& \cals_f \ar[d] \ar[r] & 0\\
            0 \ar[r] & \calh_e \ar[r, "\iota_e"] \ar[d] & \calb_e \ar[d] \ar[r, "\pi_e"] & \cals_e \ar[d] \ar[r] & 0\\
            0 \ar[r] & \calh_v \ar[r, "\iota_v"] & \calb_v \ar[r, "\pi_v"]& \cals_v \ar[r] & 0
        \end{tikzcd}
    \end{equation}
    where the vertical maps are the cosheaf extension maps. The top left square of the diagram trivially commutes to 0. Over the lower left square we find for any scalar $\dot{\theta}\in \calh_e = \R$,
    \begin{small}\begin{equation*}
        \left(\calb_{e\rhd v} \iota_e\right)
        (\dot{\theta}_e)
        =
        \dot{\theta}_e
        \begin{bmatrix}
            p_e \\
            (l_e)_\times(p_v - p_e)
        \end{bmatrix}_v
        = \dot{\theta_e}
        \begin{bmatrix}
            p_e \\
            0
        \end{bmatrix}_v
        =
        (\iota_v \calh_{e\rhd v})
        (\dot{\theta}_e)
    \end{equation*}\end{small}

    The top right square of diagram~\eqref{eq:comm_diag} commutes by construction of the extension map $\cals_{f \rhd e}$. For the lower right square, we find
    \begin{small}\begin{equation*}
        \left(\cals_{e\rhd v}  \pi_v\right) 
        \begin{bmatrix}
            \omega\\
            \beta
        \end{bmatrix}_e
        =
        \begin{bmatrix}
            [\omega] \\
            \beta + \omega_\times(p_v - p_e)
        \end{bmatrix}_v \\
       =
        \begin{bmatrix}
            \omega \\
            \beta + \omega_\times(p_v - p_e)
        \end{bmatrix}_v
        =
        \left(\pi_e  \calb_{e\rhd v} \right)
        \begin{bmatrix}
            \omega \\
            \beta
        \end{bmatrix}_e      
    \end{equation*}\end{small}
    where the middle equality follows from equation~\eqref{eq:quotient_reduces}.

    Lastly, the rows of diagram~\eqref{eq:comm_diag} are exact. Clearly $\iota$ is injective and $\pi$ are surjective. The pair $\iota$ and $\pi$ satisfy $\im \iota_e = \ker \pi_e \iso \text{span}\{l_e\}$ at each edge and $\im \iota_v = \ker \pi_v \iso so(3)$ at each vertex.
\end{proof}

The map $\iota:\calh_e \to \calb_e$ on edge stalks is the hinge matrix, described in Section~\ref{sec:classical}. In coordinates this converts hinge angular velocities to spatial velocity vectors.

From the exact sequence of cosheaves~\eqref{eq:origami_ses} we can write its long exact sequence of homology~\eqref{eq:les}. The most important segment is the beginning of the sequence:
\begin{equation}\label{eq:kinematic_les}
    0 \to H_2 \calb \xrightarrow{\pi_*} H_2 \cals \xrightarrow{\vartheta} H_1 \calh \xrightarrow{\iota_*} H_1 \calb \to \dots
\end{equation}

The connecting homomorphism $\vartheta: H_2 \cals \to H_1 \calh$ takes a spatial velocity system and converts spatial vectors to hinge angular velocities. This can be presented as a size $|E_{\text in}| \times 6|F|$ matrix that takes value at a block for incident cells $e\lhd f$:
\begin{small}\begin{equation*}
    \left(\vartheta \begin{bmatrix}
        \omega \\
        \beta
    \end{bmatrix}_f\right)_e
    = 
    [e:f] \langle l_e, \omega \rangle \in \calh_e
\end{equation*}\end{small}
where the sign follows from orientations. A full system of edge rotations $\vartheta \nu$ with $\nu\in H_2 \cals$ is a solution of the hinge model in $H_1 \calh$. The map $\vartheta$ is sketched in Figure~\ref{fig:connecting_hom}.

\begin{figure}
    \centering
    \includegraphics[scale = 1.2]{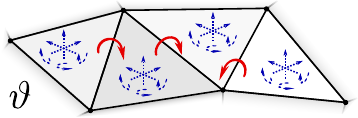}
    \caption{The connecting homomorphism $\vartheta$ sends solutions of the spatial model $\nu\in H_2 \cals$ to solutions of the hinge model $H_1 \calh$ (consisting of angular velocities $\dot{\theta}$). Global velocity vectors at all faces (thought of as elements of $H_2 \calb$) do not bend joints, so $\vartheta$ sends these to zero.}
    \label{fig:connecting_hom}
\end{figure}

\begin{theorem}[Hinge to spatial model]
\label{thm:holes}
    Suppose $(X, p)$ is any oriented origami surface. Then the map $\vartheta^+: \ker \iota_* \to H_2 \cals / se(3)$ is an isomorphism between solutions of the hinge model satisfying $6\dim H_1 X$ additional constraints $\ker \iota_* \subset H_1 \calh$, and solutions to the spatial model up to global spatial velocities $H_2 \cals/ se(3)$. These constraints maintain that there is no accumulated rotation or translation error around any homology loop of $H_1 X$.
\end{theorem}
\begin{proof}
    We know there are isomorphisms $se(3)\iso H_2 \overline{se(3)} \iso H_2 \calb \iso \pi_* H_2 \calb \subset H_2 \cals$. Applied to exact sequence~\eqref{eq:kinematic_les}, the map $\vartheta$ descends to an isomorphism
    \begin{equation}\label{eq:vartheta_iso}
        \vartheta: H_2 \cals/se(3) \to \ker \iota_*.
    \end{equation}
    which then means its Moore-Penrose pseudoinverse $\vartheta^+$ is also an isomorphism. The space $H_1 \calb$ has six dimensions per homology class of $H_1 X$. If $\iota_* (\dot{\theta})$ is the zero class, it must satisfy these $6\dim H_1 X$ additional constraints. Then there must be an element $y\in C_2 \calb$ such that $\bdd (y) = \iota (\dot{\theta})$, in other words, $y$ is a physically realizable kinematic system.
\end{proof}
If $X$ is simply connected, $H_1 X = 0$ and so $H_1 \calb = 0$. In this case line~\eqref{eq:vartheta_iso} is the isomorphism $H_2 \cals/se(3) \iso H_1 \calh$. If $X$ is an origami sheet with $m$ holes, Theorem~\ref{thm:holes} indicates there are $6m$ additional constraints ensuring three degrees of rotational and three degrees of translational alignment around the hole. For an origami surface of genus $g$ there are $12g$ additional constraint equations\footnote{
The additional $6\dim H_1 X$ constraints do not necessarily constitute a minimal set for a hinge model solution to be physical. As can be seen by exactness of sequence~\eqref{eq:kinematic_les}, some of these constraints instead link to degrees of $H_1 \cals$, as $H_1\calb \iso \im \iota_* \oplus \im \pi_*$. This is outside the scope of discussion here.
}.

\subsection{Relating the Spatial and Truss Systems}
\label{sec:truss_kinematic}

We next prove the equivalence between solutions of the spatial and truss kinematic models. Vertices are free to move while constraints lie over (higher-dimensional) edges. This is in contrast with the hinge and spatial models where the constraints were in lower dimensions, indicating that the truss model can be algebraically formulated as a {\it cellular sheaf} instead of a cosheaf\footnote{
Indeed, this {\em truss/linkage sheaf} can shown to be the dual of the {\em force cosheaf} characterizing axially loaded trusses in previous work \cite{TowardsCooperband2023}.
}. However, this abstraction is not necessary for following discussion.

Recall that an origami surface $X$ is embellished with additional vertices and edges to a stiffened linkage $X'$, described Section~\ref{sec:truss_model}. Solutions of the truss model are elements of the kernel of a size $|E'| \times 3|V'|$ block matrix we will call $\bf{M}'$. For $y\in \R^{3|V'|}$, the value of ${\bf M}'y$ at an edge $e$ with incidences $u,v\lhd e$ is
\begin{equation}\label{eq:coboundary}
    ({\bf M}'y)_e := \pm\langle l_e, y_u - y_v \rangle,
\end{equation}
the component length of $y_u-y_v$ parallel to the the edge $e$ where the sign detects orientation alignment. For $y$ to be in the kernel of ${\bf M}'$, line~\eqref{eq:coboundary} must be zero everywhere, matching the previously found constraint equation~\eqref{eq:edge_constraint}. We show an equivalence between the spatial and truss models by constructing an isomorphism $\eta: H_2 \cals \to \ker {\bf M}'$.

\begin{lemma}\label{lem:eta_def}
    Maps of the form $\eta_{f\rhd v}: \cals_f \to \R^3$ defined in line~\eqref{eq:eta} extend to a linear map $\eta: H_2 \cals \to \R^{3|V'|} = \text{Domain}({\bf M}')$.
\end{lemma}
\begin{proof}
    First fix a face $f$ incident to vertex $v$. By its construction in line~\eqref{eq:eta} $\eta$ is independent of choice of incident edge $v\lhd e_1, e_2 \lhd f$, as by cosheaf functoriality
    \[ \cals_{e_1\rhd v}\circ \cals_{f \rhd e_1} = \cals_{e_2\rhd v}\circ \cals_{f \rhd e_2}.\]
    
    Fix a homology class $\nu\in H_2 \cals$ and suppose that $v\lhd e \lhd f,g$ are two faces incident to vertex $v$. Then for $\nu$ to be a cycle it must be true that $\nu_f - \nu_g = [\omega_f - \omega_g, \beta_f - \beta_g]^\top$ is an axial rotation about the hinge $e$ and lies in $\text{span}\{[l_e]_\times\}$. Consequently $(\omega_f - \omega_g) \times (p_v - p_e) = 0$ and
    \[\cals_{f \rhd e} \nu_f - \cals_{g \rhd e} \nu_g = 0\]
    proving that $\eta$ is independent of choice of face incident to $v$. It is easy to see that $\eta$ can also be extended to the added vertices $f'$, so $\eta$ is a well defined map to the domain of ${\bf M}'$.
\end{proof}

\begin{figure}
    \centering
    \includegraphics[scale = 1.2]{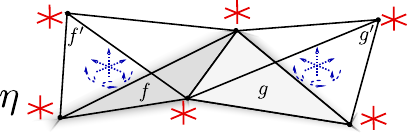}
    \caption{The linear map $\eta$ sends solutions of the spatial model $H_2 \cals$ (consisting of $se(3)$ assignments to face centroids $p_f$ satisfying constraints) to solutions of the truss model $\ker {\bf M}'$ (consisting of $\R^3$ motion assignments to vertices leaving edge lengths constant). This map $\eta$ is always an isomorphism.}
    \label{fig:enter-label}
\end{figure}

Now that $\eta$ is well defined, we show that $\eta$ extends actually maps to the kernel of ${\bf M}'$ and is both injective and surjective onto this codomain.

\begin{theorem}[Spatial to truss model]\label{thm:eta_hom}
    The linear map $\eta$ of Lemma~\ref{lem:eta_def} is an isomorphism $\eta: H_2 \cals \to \ker {\bf M}'$ between solutions of the spatial model and solutions of the truss model.
\end{theorem}
\begin{proof}
    First we show $\eta$ maps to $\ker {\bf M}'$ by showing ${\bf M}'\circ \eta = 0$. Suppose $\nu\in H_2 \cals$ is a cycle and that $u,v \lhd e \lhd f$ is a edge in the original surface $X$. Setting $\nu_f = [\omega, \beta]^\top$
    for angular velocity vector $\omega\in \R^3$ we have
    \[  (\eta \nu_f)_v = \beta + [\omega]_\times(p_v - p_f)
        \qquad
        (\eta \nu_f)_u = \beta + [\omega]_\times(p_u - p_f)\]
    and the value of ${\bf M}' \circ \eta (\nu)$ at the edge $e$ is
    \begin{equation}\label{eq:is_cycle}
            ({\bf M}' \circ \eta (\nu_f))_e
            = \pm \langle l_e, \eta(\nu_f)_v - \eta(\nu_f)_u \rangle
            = \pm \langle l_e, \omega\times (p_v - p_u) \rangle
            = 0 
    \end{equation}
    where the last equality follows because $(p_v - p_u)$ is a multiple of $l_e$. In other words, an isometry acting on a face $f$ does not change the lengths of edges along its boundary (clearly). The edge lengths of additional edges in $X'$ are not changed showing that ${\bf M}' \circ \eta = 0$ holds true over the extended complex $X'$. Thus $\eta$ maps to $\ker {\bf M}'$.

    Suppose $\eta (\nu)$ is the zero vector in $\R^{3|V'|}$; for $\eta$ to be injective we show that $\nu=0$. This is equivalent to saying that if all vertices of an origami are fixed then the faces are fixed. From the computation in equation~\eqref{eq:is_cycle} we know $\omega_e$ is a multiple of $l_e$. We are assuming the added vertex $f'$ in $X'$ has $(p_{f'} - p_f)$ linearly independent to the span of $(p_v-p_f)$ and $(p_u - p_f)$. Consequently the equation $\eta(\nu)_{f'} = 0$ implies both $\beta$ and $\omega$ are zero. This logic applies at every face.

    To see that $\eta$ is surjective, fix some motion of the extended truss $y\in \ker {\bf M}'$. By construction, the isolated sub-graph of $X'$ containing the vertices $v_i \lhd f$, the vertex $f'$, and their incident edges is rigid in $\R^3$. There then exists an assignment $\xi_f$ to each face comprised of the translation and rotation value induced by $y$ at the point $p_f$ ($\xi_f$ is an infinitesimal isometry determined by $y$) \cite{Angeles2014}. This chain $\xi\in C_2 \cals$ satisfies $\eta \xi = y$.
    
    We must now show that $\xi$ is a cycle in $H_2 \cals$. Suppose $u,v\lhd e\lhd f,g$ is an edge in $X$ with endpoints and faces. Let $k_e$ denote the difference
    \[ k_e = \cals_{f \rhd e} \xi_f - \cals_{g\rhd e} \xi_g \in \cals_e.\]
    where in spatial coordinates we have:
    
    \begin{small}\begin{equation*}
        \xi_f =
        \begin{bmatrix}
            \omega \\
            \beta 
        \end{bmatrix}_f; \qquad
        \xi_g = 
        \begin{bmatrix}
            \sigma \\
            \upsilon
        \end{bmatrix}_g \quad \text{implies} \quad
        k_e = 
        \begin{bmatrix}
            [\omega - \sigma] \\
            \beta - \upsilon + \omega_\times(p_e-p_f) - \sigma_\times(p_e - p_g)
        \end{bmatrix}_e
    \end{equation*}\end{small}
    
    Then $\cals_{e\rhd v} k_e$ is a translation of the node $v$ equal to
    \[\cals_{e\rhd v} k_e = \cals_{e\rhd v} \circ \left(\cals_{f \rhd e}\xi_f - \cals_{g\rhd e} \xi_g\right) = (\eta \xi_f - \eta \xi_g)_v = y_v - y_v = 0\]
    and likewise $\cals_{e\rhd u} k_e = 0$. Examining the bottom three velocity coordinates of $\cals_{e\rhd v} k_e$ and $\cals_{e\rhd u} k_e$ we find the equations
    \begin{equation}\label{eq:bddkv}
        (\omega - \sigma)_\times (p_v-p_e) + \beta - \upsilon + \omega_\times(p_e-p_f) - \sigma_\times(p_e - p_g) = 0
    \end{equation}
    \begin{equation}\label{eq:bddku}
        (\omega - \sigma)_\times (p_u-p_e) + \beta - \upsilon + \omega_\times(p_e-p_f) - \sigma_\times(p_e - p_g) = 0
    \end{equation}
    Subtracting equation~\eqref{eq:bddku} from equation~\eqref{eq:bddkv} we have the equation $(\omega - \sigma)_\times(p_v-p_u) = 0$ and so $(\omega - \sigma)\in \text{span}\{l_e\}$. Utilizing this fact in equation~\eqref{eq:bddkv} we find that both the angular and linear velocity components of $k_e$ are zero, hence $k_e$ is likewise zero. Since the edges and faces $e\lhd f,g$ are general, it follows that $\bdd_\cals (\xi) = 0$ and $\xi$ is a homology cycle.
\end{proof}

Lemma~\ref{lem:eta_def} and Theorem~\ref{thm:eta_hom} are purely local statements and are unaffected by global structure, hence they apply to any origami surface topology. This theorem combined with the prior Theorem~\ref{thm:holes} leads to the main result of Corollary~\ref{thm:full}.

\section{Conclusion}

We have described the linear relationship between solutions to the rotating hinge, spatial, and truss models over any oriented origami surface. This procedure allows for the recovery of the full kinematic data (including translations at faces, edges, and nodes) from only the rotation value of each hinge. The authors would like to extend these methods towards a homological description of (closed-chain) inertial origami dynamics, building upon the tree-augmented approach and spatial operator algebra in \cite{FULTON2021549}. This would involve a fully homological forward kinematics algorithm, not unlike that described in~\cite{jain2010robot}. Moreover, more work must be done to connect the cosheaf ``tangent space'' representations presented with the full configuration space, akin from moving from the Lie algebra to the Lie group as discussed in~\cite{muller2018screw}.

In the Introduction~\ref{sec:intro}, we posited that serial chains and tree-type systems are simple because their underlying topologies are contractible. Making this precise, there is a {\it discrete Morse function} $t: X \to \R$ on cells setting $t(f_i) = i$ and $t(e_i) = i$~\cite{forman2002user}. The induced discrete gradient field flows from the end of the serial chain back to the base, reverse of the directed graph in Figure~\ref{fig:serial} (right). This leads to pairings of adjacent cells $(e_i, f_i)$ that preemptively reduce homology dimensions and ``contracting'' along the flow line. In the computation of $H_2 \cals$ for instance, this pairs a subspace of $\cals_{f_i}$ with $\cals_{e_i}$, reducing the computational cost of later linear algebraic steps. These future ideas highlight the potential for kinematic complexity to be reduced via homological algorithms.

\section*{Acknowledgements}

The authors acknowledge the support of the Air Force Office of Scientific Research under award number FA9550-21-1-0334. The authors are grateful for their illuminating discussions with Professors Tomohiro Tachi and Thomas Hull.

\printbibliography


\appendix

\section{Basic Homological Algebra}
\label{app:homological}

Homological algebra is an extension of linear algebra. In linear algebra, the focus is on linear transformations between vector spaces -- their decompositions and characteristics. Homological algebra considers {\it systems} of linear transformations over sequences, grids, surfaces, and more. It is thus an ideal construct for serial-chain and origami structures. The standard references in homological algebra \cite{Weibel1994} tend to be more abstract, using more abstract algebra and category theory. This brief primer utilizes simple finite-dimensional real vector spaces for concreteness.

Cellular cosheaves are comprised of distributions of interconnected algebraic and geometric information across cells, along with methods for determining the resulting equilibrium. Underlying the abstract notation is nothing more than vector spaces and linear maps.

\begin{definition}\label{def:cosheaf}
    A (finite vector valued cellular) cosheaf $\calf$ over a cell complex\footnote{
    Here $X$ is a two-dimensional non-degenerate oriented polyhedral surface. See \cite{Curry2013} for the full definition and requirements.}
    $X$ is comprised of the following data. To each cell $c$ in $X$ a finite-dimensional vector space is assigned $\calf_c$ called a {\it stalk}. Where cells are incident $c \lhd d$ a linear {\it cosheaf extension map} is assigned between stalks $\calf_{d \rhd c}: \calf_d \to \calf_c$.
\end{definition}

Cosheaves are {\it functors} from the poset category of $X$ to the category of finite-dimensional vector spaces. This means $\calf_{c\rhd b} \circ \calf_{d\rhd c} = \calf_{d\rhd b}$ and $\calf_{c\rhd c} = \id$ for any incident cells $b\lhd c \lhd d$.

The cosheaf definition is extremely flexible. In practice, a cosheaf is {\it programmed} by choosing particular vector spaces for stalks and particular extension maps. To calculate global properties of a cosheaf, one combines the dispersed stalks to a single structure called a {\it chain complex}, a core construct in homological algebra.

\begin{figure}
    \centering
    \includegraphics[scale = 1.2]{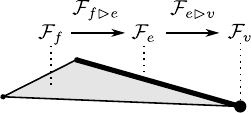}
    \caption{A sketch of a cosheaf over adjacent cells. The stalks $\calf_c$ are vector spaces and the extension maps $\calf_{d\rhd c}$ are linear. There are stalks over other cells in the triangle not explicitly labeled.}
    \label{fig:cosheaf_sketch}
\end{figure}

\begin{definition}
An $i$-dimensional chain of a cosheaf $\calf$ is the combined sum of vectors $x = \sum x_d$ across all stalks $x_d\in \calf_d$ over cells $d$ of dimension $i$. The boundary of $x$ is an $(i-1)$-dimensional chain $\bdd x$ taking value
\begin{equation}\label{eq:bdd}
    (\bdd x)_c := \sum_{c \lhd d} \pm \calf_{d \rhd c} x_d
    \quad\text{ for }\quad
    x\in C_i \calf := \bigoplus_{\dim d = i} \calf_d
\end{equation}
at an $(i-1)$-dimensional cell $c$ where the sign $\pm 1$ detects whether local orientations of the cells $c$ and $d$ agree or disagree. This operation forms a linear {\it boundary map} $\bdd$ between spaces of chains
\begin{equation}\label{eq:chain_complex}
    \dots \to C_2 \calf \xrightarrow{\bdd_2} C_1 \calf \xrightarrow{\bdd_1} C_0 \calf \to 0.
\end{equation}
\end{definition}
    
One can show that $\bdd \circ \bdd = 0$, indicating that the above construction forms a {\it chain complex} ubiquitous in algebraic topology \cite{AlgebraicHatcher2002}. The {\it homology} of a cosheaf $\calf$ is the homology of its chain complex~\eqref{eq:chain_complex}
\begin{equation}\label{eq:homology}
    H_i \calf := \ker \bdd_i / \im \bdd_{i+1}.
\end{equation}
This quotient vector space consists of {\it cycles} (elements of $\ker \bdd_i$) modulo the boundaries of higher-dimensional chains.

\begin{example}\label{ex:constant}
   The most elementary example is the constant cosheaf $\overline{W}$ associated to a finite-dimensional vector space $W$. Over each cell $c$ in $X$ we assign the stalk $\overline{W}_c = W$. All extension maps $\overline{W}_{d\rhd c}$ are the identity transformation. Constant cosheaves capture nothing but the underlying topology; its homology $H_i \overline{W}$ is isomorphic to $\dim W$ copies of the base homology $H_i X$. Standard homology follows when $W=\R$; in this case, the boundary matrix $\bdd_1$ is better known as the {\it signed incidence matrix} in graph theory.
\end{example}

\begin{definition}
    Between two cosheaves $\calf$ and $\calg$ over a cell complex $X$, a {\it cosheaf map} $\varphi: \calf \to \calg$ is comprised of stalk wise component linear maps $\varphi_c: \calf_c \to \calg_c$. The cosheaf map $\varphi$ must additionally satisfy a {\it naturality condition} $\varphi_c \circ \calf_{d\rhd c} = \calg_{d \rhd c} \circ \varphi_d$ over every incident pair of cells $c \lhd d$.
\end{definition}

Often one wishes to consider a cosheaf consisting of a subset of the data in another cosheaf. In this scenario we define an injective cosheaf map $\varphi: \calf \to \calg$, where each leg $\varphi_c$ of the map is an embedding. The complimentary data structure of $\calg$ with respect to $\calf$ is captured in the {\it quotient cosheaf}. Over each cell, we have the data space $\calg_c / \varphi_c\calf_c$ (abbreviated to $\calg_c / \calf_c$), assembling with quotient extension maps.

\begin{definition}\label{def:ses}
    A {\it short exact sequence of cosheaves} is a sequence of cosheaf maps
    \begin{equation}\label{eq:cosheaf_ses}
        0 \to \calf \xrightarrow{\varphi} \calg \xrightarrow{\rho} \calg / \calf \to 0
    \end{equation}
    where at each cell $c$ the stalks form a short exact sequence of vector spaces
    \begin{equation}\label{eq:vec_ses}
        0 \to \calf_c \xrightarrow{\varphi_c} \calg_c \xrightarrow{\rho_c} (\calg / \calf)_c \to 0.
    \end{equation}
    at every stalk. Exactness means at every place of sequence~\eqref{eq:vec_ses} the image of the incoming map ($\varphi_c$)is equal to the kernel of the outgoing map ($\rho_c$).
\end{definition}
Every short exact sequence of cosheaves~\eqref{eq:cosheaf_ses} induces a short exact sequence of cosheaf chain complexes. Leveraging classical homological algebra, from this follows a {\it long exact sequence in homology}
\begin{equation}\label{eq:les}
    \dots \to H_{i+1} \calg / \calf \xrightarrow{\vartheta} H_i \calf \xrightarrow{\varphi_*} H_i \calg \xrightarrow{\rho_*} H_i \calg / \calf \xrightarrow{\vartheta} H_{i-1} \calf \to \dots
\end{equation}
where the linear maps $\vartheta$ are called {\it connecting homomorphisms} and $\varphi_*,\iota_*$ are induced maps on homology. Proof of the existence and exactness of the long exact sequence of homology can be found in any introductory algebraic topology textbook \cite{AlgebraicHatcher2002}.

\end{document}